\documentclass[a4paper,11pt]{amsart}

\usepackage{amssymb,xspace,amscd,yhmath,mathdots,txfonts,
mathrsfs}

\usepackage[matrix,arrow,curve]{xy}\CompileMatrices

\usepackage{hyperref}
\usepackage{yfonts}[1998/10/03]

\numberwithin{equation}{section}

\theoremstyle{plain}

\newtheorem{thm}{Theorem}[section]

\newtheorem{lem}[thm]{Lemma}

\newtheorem{cor}[thm]{Corollary}

\newtheorem{prop}[thm]{Proposition}

\theoremstyle{definition}

\newtheorem{dfn}[thm]{Definition}
\newtheorem{defn}[thm]{Definition}

\newtheorem{rmk}[thm]{Remark}

\DeclareMathOperator{\cD}{\mathcal{D}}
\DeclareMathOperator{\cO}{\mathcal{O}}
\DeclareMathOperator{\cW}{\mathcal{W}}
\DeclareMathOperator{\cU}{\mathcal{U}}
\DeclareMathOperator{\R}{\mathbb{R}}
\DeclareMathOperator{\cC}{\mathcal{C}}

\DeclareMathOperator{\Zt}{{\mathscr{Z}}}
\DeclareMathOperator{\At}{{\mathscr{A}}}
\DeclareMathOperator{\Ht}{{\mathscr{H}}}
\DeclareMathOperator{\Kt}{{\mathscr{K}}}
\DeclareMathOperator{\Mt}{{\mathscr{M}}}

\DeclareMathOperator{\C}{\mathbb{C}}
\DeclareMathOperator{\tO}{\textswab{O}}
\DeclareMathOperator{\wf}{\mathrm{WF}}
\DeclareMathOperator{\awf}{\overline{\mathrm{WF}}}

\title[$\mathcal{C}^{\infty}$-hypoellipticity and extension of $CR$ functions]
{$\mathcal{C}^{\infty}$-hypoellipticity and extension of $CR$ functions}

\author[M.~Nacinovich]{Mauro Nacinovich}
\address{M.\ Nacinovich:
Dipartimento di Matematica\\ II Universit\`a di Roma
``Tor Ver\-ga\-ta''\\ Via della Ricerca Scientifica\\ 00133 Roma
(Italy)}
\email{nacinovi@mat.uniroma2.it}

\author[E.~Porten]{Egmont Porten}
\address{E.\ Porten: Department of Mathematics\\ 
Mid Sweden University\\ 85170 Sundsvall \\Sweden}
\email{Egmont.Porten@miun.se}

\date{July 10, 2011}

\subjclass[2000]{Primary: 32V20
Secondary: 32V05, 32V25, 32V30, 32V10, 32W10, 32D10, 35H10, 35H20, 35A18, 35A20,
35B65, 53C30}
\keywords{$CR$-hypoelliptic, $CR$-embedding, holomorphic extension,
$\cC^\infty$ wave front set, holomorphic wedge extension}

\begin{document}
\maketitle
\tableofcontents
\begin{abstract}
Let $M$ be a $CR$ submanifold of a complex manifold $X$.
The main result of this 
article is to show that $CR$-hypoellipticity at $p_0\in{M}$ is
necessary and sufficient for holomorphic extension 
of all germs of $CR$ functions
to an ambient
neighborhood in $X$. As an application, we obtain that 
$CR$-hypoellipticity implies  
the existence of generic embeddings and prove holomorphic extension
for a large class of 
$CR$ manifolds satisfying a higher order 
Levi pseudoconcavity condition.
\end{abstract}
\section{Introduction}
Let $M$ be an abstract $CR$ manifold, of arbitrary $CR$ dimension 
$m$ and
$CR$ codimension $d$.  
We say that $M$ is \emph{$CR$-hypoelliptic} at
$p_0\in{M}$ if every distribution satisfying the 
homogeneous tangential Cauchy-Riemann
equations on a neighborhood of $p_0$ in $M$ is 
$\cC^\infty$-smooth on a neighborhood of $p_0$. \par
A \emph{local $CR$-embedding} of $M$ at $p_0$ is the datum of
$\cC^\infty$-smooth solutions $z_1,\hdots,z_{\nuup}$ to the homogeneous
tangential Cauchy-Riemann equation on a neighborhood $U$ of $p_0$ in
$M$ such that the map $p\mapsto (z_1(p),\hdots,z_{\nuup}(p))$ is a
smooth embedding $U\hookrightarrow\C^{\nuup}$. We have $\nuup\geq{m+d}=n$,
and when we have equality we say that the local $CR$-embedding is
\emph{generic}. \par
Note that from any local $CR$-embedding we can obtain a generic local
$CR$-embedding of a smaller neighborhood of $p_0$,
by choosing any subset $z_{i_1},\hdots,z_{i_n}$ of $z_1,\hdots,z_{\nuup}$
with $dz_{i_1}(p_0)\wedge\cdots\wedge{d}z_{i_n}(p_0)\neq{0}$. \par
We say that $M$ has the 
\emph{holomorphic extension property} at $p_0$
if there is a generic local $CR$-embedding $\phiup:U\hookrightarrow\C^n$ 
such that,
for every distribution solution $u$ of the homogeneous tangential 
Cauchy-Riemann equations on a neighborhood $U'\subset{U}$ of $p_0$ in $M$,
there is a holomorphic function, defined on a neighborhood $V$ of
$\piup(p_0)$ in $\C^n$, such that $\phi^*\tilde{u}$ is defined and equal
to $u$ on a neighborhood of $p_0$ in $U'$. \par
We can also consider weaker formulations of the holomorphic extension
property, either by dropping the assumption that the local $CR$-embedding
$\phiup$ be generic, or allowing different embeddings for extending
different $CR$-distributions, or keeping a same local $CR$-embedding
but requiring local holomorphic extension only for smooth $CR$-functions. 
\par
The fact that the different formulations are in fact equivalent is
a consequence of our main result:
\begin{thm}\label{main} Let $M$ be a $CR$ manifold, locally $CR$-embeddable
at $p_0\in{M}$. Then $M$ 
has the holomorphic extension property at
$p_0$ if and only if $M$ is $CR$-hypoelliptic at $p_0$.
\end{thm}
An interesting consequence of Theorem\,\ref{main} is
a uniqueness result for the local $CR$-embedding of $M$ at $p_0$:
\begin{cor}\label{cor12}
Under the assumption of Theorem\,\ref{main} we have:
\begin{enumerate}
\item If $p_0\in{U}^{\text{open}}\subset{M}$ and 
$\phiup:U\to\C^{\nuup}$ is a local $CR$-embedding, then there is
an $n$-dimensional complex submanifold $X$ of an open neighborhood 
$V$ of $p_0$ in $\C^{\nuup}$ and $\omegaup$, with $p_0\in\omegaup^{\text{open}}
\subset{U}$ such that $\phiup(\omegaup)\subset{X}$. 
\item If $p_0\in{U}^{\text{open}}\subset{M}$ and 
$\phiup_i:U\to\C^n$, for $i=1,2$, are two generic local $CR$-embeddings with
$\phiup_i(p_0)=0$, then there are open neighborhoods $V,W$ of $0$ in
$\C^n$ and a biholomorphic map $\psiup:V\to{W}$ such that 
$\phiup_2=\psiup\circ\phiup_1$ on~$V$.
\end{enumerate}
\end{cor}
This corollary has the consequence that, when $M$ is
$CR$-hypoelliptic and locally embeddable
at all points, its $CR$ structure completely determines its
hypo-analytic structure (see \cite{t92}). Moreover, the arguments of
\cite{AF79} also yield
\begin{cor}\label{cor14} Let $M$ be a $CR$ manifold of $CR$ dimension $m$ and 
$CR$ codimension $d$, and $n\! =\! m\! + \! d$. 
Assume that $M$ is locally $CR$-embeddable and  $CR$-hypoelliptic at
all points. Then $M$ admits a smooth generic $CR$-embedding
$M\hookrightarrow{X}$ into an $n$-dimensiona complex manifold $X$.
\par
If $\phiup_i:M\to{X}_i$, $i=1,2$, are two generic $CR$-embeddings of $M$, 
then there are tubular neighborhoods $Y_i$ of $\phiup_i(M)$ in $X_i$, $i=1,2$,
and a biholomorphic map $\psiup:Y_1\to{Y}_2$, such that $\psiup(\phiup_1(M))=
\phiup_2(M)$. 
\end{cor}
Since holomorphic functions are real-analytic,
holomorphic extendability trivially 
implies $CR$-hypoellipticity. 
Both $CR$-hypoellipticity and holomorphic extendability
imply minimality.
The main result of this note 
is that {$CR$-}hypoellipticity and holomorphic 
extendability  are equivalent at minimal points. \par
Since real-analytic $CR$ manifolds are locally $CR$-embeddable 
(see \cite{AF79}), and holomorphic functions are real-analytic,
we obtain
\begin{cor}\label{cor15}
Assume that $M$ is a real-analytic $CR$-manifold, and let $p_0\in{M}$. Then
the following are equivalent:
\begin{enumerate}
\item $M$ is $CR$-hypoelliptic at $p_0$;
\item  $M$ is $CR$-analytic-hypoelliptic at $p_0$;
\item all smooth solutions of the homogeneous tangential Cauchy-Riemann
equations on a neighborhood of $p_0$ are real-analytic at $p_0$.
\end{enumerate}
\end{cor}
We also point out that our result applies to give concrete applications for
the Siegel-type theorems proved in \cite{NH04a,NH05a}
about the trancsendence degree of the fields of $CR$-meromorphic functions.\par
Despite of several contributions, the problem of finding 
a geometric characterization for the holomorphic extension property
is still wide open, even for real analytic hypersurfaces. 
The interest of Theorem \ref{main} is that it establishes a
link between holomorphic extension and $\cC^{\infty}$ regularity, 
a central and better understood topic in PDE theory.
We illustrate this point of view by recalling in \S\ref{sec6}
the weak pseudoconcavity assumptions of \cite{AHNP1},
generalizing the \textit{essential
pseudoconcavity} of \cite{HN00},
which insure $CR$-hypoellipticity, and illustrating 
by some examples 
in \S\ref{sec7} how this approach leads to the proof of the
holomorphic extension property for manifolds with a highly 
degenerate Levi form. Extension theorems had been obtained
before under stronger non-degeneracy assumptions on the Levi form
(see e.g. \cite{BP82, naci-va}), or for $CR$ manifolds satisfying
a third order pseudoconcavity condition (see \cite{AHNP2}).\par
We notice that minimality is a necessary condition for 
$CR$-hypoellipticity by \cite{ba-ro}, and that some sort of
pseudoconcavity is also necessary, as holomorphic extension
does not hold e.g. when $M$ lies in the boundary of a
domain of holomorphy. 
\par
In general, germs of $CR$ functions on a generically embedded $CR$ manifold
$M\hookrightarrow{X}$ may fail to holomorphically extend to a full neighborhood
$\cU$ of $p_0$ 
in $X$ and one can consider instead open subsets
$\cW$ of $X$ for which $M\cap\partial\cW$ is a neighborhood of $p_0$ in $M$.
A fundamental 
result of Tumanov \cite{Tu1} states that holomorphic local 
\emph{wedge} extension
is valid if $M$ is minimal at $p_0$. By \cite{ba-ro}, 
this condition is also necessary. However,
 the known proofs of local holomorphic wedge extension merely yield existence, 
but no explicit information on its shape.
The analytic or hypo-analytic wave front sets tautologically
give the directions of holomorphic extension. We conjecture that,
in analogy with Theorem\,\ref{main}, the union  of the
$\mathcal{C}^{\infty}$ wave front sets of all germs of $CR$ distributions
and that of their hypo-analytic wave front sets coincide. 
Theorem\,\ref{wf} in \S\ref{sec5} is a first partial 
result in this direction.
\par\smallskip
Let us shortly describe the contents of the paper. In \S\ref{sec2} 
we set notation and precise the notion of $CR$-hypoellipticity. \S\ref{sec3}
contains the proof of Theorem\,\ref{main}. In \S\ref{sec4} we prove various
equivalences of the extension property, easily implying Corollaries\ref{cor12},
\ref{cor14},\ref{cor15}. Section \S\ref{sec5} contains our result
about wedge extension and the common $\cC^\infty$ wave front set of germs
of $CR$ distributions. In \S\ref{sec6} we rehearse the subellipticity result
of \cite{AHNP1} and in \S\ref{sec7} we give some examples.



\section{$CR$-hypoellipticity} \label{sec2}
Let
$M$ be an abstract smooth $CR$ manifold
of $CR$ dimension $m$ and $CR$ codimension $d$. The $CR$ structure on $M$
is defined by the datum of an $m$-dimensional subbundle $T^{0,1}M$ of
the complexified tangent bundle $\mathbb{C}TM$ with
\begin{equation*}
  T^{0,1}M\cap\overline{T^{0,1}M}=\underline{0}\quad\text{and}\quad
[\Gamma(M,T^{0,1}M),\Gamma(M,T^{0,1}M)]\subset\Gamma(M,T^{0,1}M).
\end{equation*}
\par
For $U^{\text{open}}\subset{M}$ we denote by $\cO_M^{\infty}(U)$ the set
of smooth solutions on $U$ to the tangential Cauchy-Riemann equations:
\begin{equation*}
  \cO_M^{\infty}(U)=\{u\in\cC^\infty(U,\mathbb{C})\mid Zu=0,\;
\forall Z\in\Gamma(U,T^{0,1}M)\}.
\end{equation*}
Likewise, we denote by $\cO_M^{0}(U)$ and 
$\cO_M^{-\infty}(U)$ the spaces of complex valued continuous functions and 
of complex valued distributions, respectively, 
that weakly solve the homogeneous equations
\begin{equation*}
 \text{$Zu=0,\;\forall Z\in\Gamma(U,T^{0,1}M)$ on $U$ },
\end{equation*}
i.e. such that
\begin{equation*}
  \int u \,Z'\phi \;d\mu = 0,\quad\forall \phi\in\mathcal{C}^{\infty}_0(U),
\;\forall Z\in\Gamma(U,T^{0,1}M),
\end{equation*}
where $\mu$ is a positive measure with smooth density on $M$ and 
the formal adjoint $Z'$ of $Z\in\Gamma(U,T^{0,1}M)$ is defined by
\begin{equation*}
  \int Zv\,{\phi}\; d\mu=\int v\, {Z'\phi}\; d\mu,\quad
\forall v,\phi\in \mathcal{C}^{\infty}_0(U).
\end{equation*}
The assignments $U^{\text{open}}\to \cO_M^a(U)$, for $a=-\infty, 0, \infty$,
define sheaves of germs. We denote by 
$\cO_{M,(p_0)}^{a}$ the stalk at $p_0\in{M}$. When $M$ is a complex manifold
we drop the superscript $a$, because the three sheaves coincide by
the regularity theorem for holomorphic functions.
\begin{defn}
We say that $M$ is \emph{$CR$-hypoelliptic} at $p_0\in{M}$ if
$\cO_{M,(p_0)}^{-\infty}=\cO_{M,(p_0)}^{\infty}$.  
\end{defn}
\section{Proof of Theorem\,\ref{main}}\label{sec3}
By taking a generic $CR$-embedding, we can as well assume
that $M\subset\C^n$, where $n=m+d$, and
$m$ is the $CR$ dimension, $d$ the 
$CR$ codimension of $M$. We can also assume that
$p_0=0$ and that the holomorphic coordinates of $\mathbb{C}^n$ have
been chosen in such a way that $M$ is the graph 
\begin{equation}\label{graph}
y'=h(x',z'')
\end{equation}
of a smooth map $h:V\to{\mathbb{R}}^d$, with $h(0)=0$, $dh(0)=0$, 
for an open neighborhood
$V$ of $0$ in $\mathbb{R}^d\times\mathbb{C}^m$. Here
$z=(z',z'')\in\mathbb{C}^d\times\mathbb{C}^m$, with $d+m=n$,
and $z'=x'+iy'$, $z''=x''+iy''$ with $x',y'\in\mathbb{R}^d$,
$x'',y''\in\mathbb{R}^m$.\par
 An \textit{open wedge}
$\cW$ attached to $M$ along
an open set $E=\mbox{Edge}(\cW)\subset M$ is, in the chosen coordinates, 
a set of the form
\begin{equation}\label{wedge}
\cW=\{z+(i{x}',0):z\in E, {x}'\in \mathrm{C}\},
\end{equation}
where $\mathrm{C}\subset\R^d$ is a truncated open cone with vertex at the origin. 
Note that $\cW$ is foliated
by the approach manifolds $E_{y'}=\{z+(iy',0):z\in E\}$, $y'\in \mathrm{C}$. 
Recall that $f\in\cO(\cW)$ attains the weak boundary
values $f^*\in\cD'(E)$ along $E$ if for every test function $\phi\in\cD(E)$, we have
\begin{equation}
\lim_{h'\rightarrow 0,y'\in 
\mathrm{C}}\int f(x'+ih(x',z'')+iy',z'')\phi(x',z'')\,dm_{d+2m}= f^*[\phi].
\end{equation}
Here $dm_{d+2m}$ denotes standard Lebesgue measure on $\R^d\times\C^m$. 
A function $f\in\cO(\cW)$ has polynomial
growth along $E$ if for every compact $K\subset E$ there are an integer 
$N_K\geq 0$ and a constant $a_K>0$ such that
\begin{equation}
|f(x'+ih(x',z'')+iy',z'')|\leq a_K |y'|^{-N_K},\;
\forall (x',z'')\in K,\;\forall y'\in \mathrm{C}.
\end{equation}
Holomorphic functions of polynomial 
growth attain unique distribution boundary values on $E$,
which 
weakly satisfy the homogeneous
tangential $CR$ equations.
\begin{proof}[Proof of Theorem \ref{main}:]
Before going into the technical details of the proof, we
sketch the main ideas involved. 
As already mentioned,
we need only to show that $CR$-hypoellipticity
implies holomorphic extension 
to full neighborhoods.
First we observe that $p_0$ must be 
a minimal point of $M$. Otherwise, $M$ contains a proper
$CR$ submanifold $N$ through $p_0$, of the same $CR$ dimension.
Then a suitable distribution carried by 
$N$ would define a
non smooth $CR$-distribution on a neighborhood of $p_0$
(see \cite{tr1,ba-ro}). Thus $CR$-hypoellipticity implies minimality at
$p_0$. Hence, all 
$CR$ distributions on a neighborhood $U\subset M$ of $p_0$
are boundary values of holomorphic functions defined on
an open wedge $\cW=\cW_U$. Then we argue by contradiction,
assuming that not all $CR$ distributions 
holomorphically extend to a full neighborhood of $p_0$. 
We consider the envelope of holomorphy $X$ of $\cW$,
and identify $p_0$ to a point of its abstract boundary $bX$.
Then we construct
a holomorphic function $f$ on $X$ 
with polynomial growth on $bM$, and 
whose modulus is unbounded in any neighborhood of $p_0$.
Pushing down to $\cW$, we obtain a function 
with polynomial growth along the edge 
with a $CR$-distribution boundary value which is unbounded, and hence
discontinuous, at $p_0$.
\par
Let us choose holomorphic coordinates $(z',z'')$ centered at $p_0$ as in (\ref{graph}),
and let $U\subset M$ be an open neighborhood of $0$ which carries a $CR$ distribution
which does not holomorphically extend to an ambient neighborhood of $0$.
Since $M$ is minimal at $0$ as noticed above, Tumanov's theorem 
yields an open wedge $\cW$ as in
(\ref{wedge}) such that every $CR$ distribution on $U$ has a holomorphic extension
to $\cW$.
\par
Let $\pi:X\rightarrow\C^n$ be the envelope of holomorphy of $\cW$. 
Recall that $X$ is a Stein manifold spread
over $\C^n$ by a locally biholomorphic mapping $\pi$. 
Moreover there is a canonical injective holomorphic map
$\alpha:\cW\rightarrow X$ satisfying $\pi\circ\alpha=\mbox{id}_{\cW}$ 
such that for every $g\in\cO(\cW)$
the pushforward  $\alpha_* g$ to $\cW'=\alpha(\cW)$ extends to $X$ 
holomorphically, and such that $X$ is
a maximal Riemann domain with this property (see \cite{ja-pf,me-po1} 
for detailed information).\par
We recall the construction, due to Grauert and Remmert, 
which yields a canonical abstract closure
$\overline{\pi}:\overline{X}\rightarrow\C^n$ in the following way: 
A boundary point is
a maximal filter\footnote{A filter 
is a family of subsets such that for each pair of members 
$U_1,U_2$, there
is a third member $U_3$ with $U_3\subset U_1\cap U_2$.}
$\mathfrak{a}$ of connected open sets in $X$ such that 
\begin{enumerate}
\item[{\bf (i)}] $\mathfrak{a}$ 
has no accumulation point in $X$,
\item[{\bf (ii)}] for
every $U\in\mathfrak{a}$, there is 
$V^{\text{open}}\subset\C^n$ such that $U$ is a
connected component of $\pi^{-1}(V)$, 
\item[{\bf (iii)}] the image filter $\pi_* \mathfrak{a}$ converges 
to a point $z\in\C^n$, and
\item[{\bf (iv)}] 
for every open neighborhood $V\subset\C^n$
of $z$ one of the components of $\pi^{-1}(V)$ is a member 
of $\mathfrak{a}$. 
\end{enumerate}
We will denote the abstract
boundary of $X$ by $b X$.
Setting $\overline{\pi}(\mathfrak{a})=z$ 
in the above situation, one obtains
an extension of $\pi$ to $\overline{X}=X\cup bX$, and
there is a natural Hausdorff topology on
$\overline{X}$ such that $\overline{\pi}$ is continuous 
(see \cite{ja-pf} for the details). Note that the
topological boundary $\partial D$ of a domain $D\subset\C^n$ 
may not coincide with its abstract boundary $bD$.
\par
Our assumption that holomorphic extension to a full neighborhood of 
$0$ fails implies that the abstract boundary $bX$ contains a point
$0'$ with $\overline{\pi}(0')=0$. 
We denote by $\delta_X(p)$ the \textit{distance from the boundary}
in $X$. It can be defined by
\begin{equation*}
  \delta_X(p)=\sup\{r>0\mid \{|z-\pi(p)|<r\}\subset\pi(X)\}.
\end{equation*}
For each integer 
$k\geq 0$, we define the space of 
holomorphic functions on $X$, with $k$-polynomial growth on $bX$, by
\begin{equation*}
\cO^{(k)}(X)=\{f\in\cO(X)\mid \delta_X^k f\;\text{is bounded on $X$}\}.
\end{equation*}
It is a Banach spaces with the norm 
$\|f\|_{\cO^{(k)}(X)}=\sup_{p\in X}|\delta_X^k(p)f(p)|$.
\begin{lem}\label{lem21}
There is a sequence $\{p_j\}_{j=1,2,\ldots}\in\cW'=\alpha(\cW)$, 
satisfying $\pi(p_j)\rightarrow 0$, and
a function $f\in\cO^{(2n+1)}(X)$ such that $|f(p_j)|\rightarrow\infty$.
\end{lem}
For subdomains of $\C^n$, more 
precise results can be found in \cite{Pf}.
\begin{proof}
We will use the following result, which is a particular case of 
\cite[Proposition 2.5.4]{ja-pf}:\qquad 
\textsl{There is a constant $C>0$, only depending on $X$, such that
  \begin{equation*}
    \forall p\in{X}\;\;\exists f_p\in\mathcal{O}^{(2n+1)}(X)\;\;\text{with}\;\;
f(p)=1,\;\; \|f\|_{\mathcal{O}^{(2n+1)}(X)}\leq C\delta_X(p).
  \end{equation*}}
We will prove by induction that there are points $p_j\in X$ 
and functions $f_j\in\cO^{(2n+1)}(X)$,
$j=1,2,\ldots$, satisfying
\begin{itemize}
  \item[\bf (a)] $p_j\in\alpha(\cW\cap B_0(1/j))$,
  \item[\bf (b)] $ |f_j(p_j)|\geq j$,
  \item[\bf (c)] $\|f_j-f_{j-1}\|_{\cO^{(2n+1)}(X)}<2^{-j}$, and
  \item[\bf (d)] $\sup_{X_{\leq\delta_{j-1}}}|f_j-f_{j-1}|\leq 2^{-j}$,
\end{itemize}
where we have abbreviated $\delta_j=\delta_X(p_j)$, $X_{\leq d}
=\{p\in X:\delta_X(p)\leq d\}$.
\par
Take any $p_1\in\alpha(\cW\cap B_0(1))$ and set $f_1\equiv 1$. 
Assume by recurrence that we already found
$p_1,\ldots,p_{k-1}$ and $f_1,\ldots,f_{k-1}\in\cO^{(2n+1)}(X)$ 
satisfying 
{\bf (a)}-{\bf (d)} for $j\leq k-1$.
Choose $p_k\in\alpha(\cW\cap B_0(1/k))$ such that 
$\delta_k\leq\frac{\delta^{2n+1}_{k-1}}{k 2^k C}$.
If $|f_{k-1}(p_k)|\geq k$ holds, $f_k=f_{k-1}$ 
obviously satisfies {\bf (a)}-{\bf (d)} for $j=k$.
Otherwise, we pick a function $f_{p_k}$ as in the above-cited result and set
$f_k=f_{k-1}+k \alpha f_{p_k}$, with $\alpha=1$ if $f_{k-1}(p_k)=0$ and
$\alpha=\frac{f_{k-1}(p_k)}{|f_{k-1}(p_k)|}$ otherwise. 
This implies {\bf (b)} for $j=k$. We verify that
\[
\|f_k-f_{k-1}\|_{\cO^{(2n+1)}(X)}=k\|f_{p_k}\|_{\cO^{(2n+1)}(X)}\leq kC\delta_{k}\leq 2^{-k},
\]
and
\[
\sup_{X_{\leq\delta_{k-1}}}|f_k-f_{k-1}|=k\sup_{X_{\leq\delta_{k-1}}}|f_{p_k}|\leq
\frac{C}{\delta^{2n+1}_{k-1}}\sup_{X_{\leq\delta_{k-1}}}|\delta^{2n+1}_X f_{p_k}|\leq 2^{-k},
\]
completing the inductive step.

Now {\bf (c)} implies that the $\cO^{(2n+1)}(X)$-limit $f=\lim f_m$ 
exists, and {\bf (b)}, {\bf (d)}
yield $|f(p_j)|\geq j-1$ for all $j$. The proof is complete.
\end{proof}
The push forward $f\circ\alpha$  
of the function $f$ 
obtained in Lemma\,\ref{lem21} is holomorphic on $\cW$ and 
has polynomial growth while approaching the edge $E$ of $\cW$, because
$E\subset\overline{\pi}(\overline{X})$. 
In particular, 
$f\circ\alpha$ 
has a boundary value, which is a
$CR$ distribution
$f^*$ on $E$. 
By \cite[Lemma 7.2.6]{BER}, 
$f$ is continuous up to the edge near
every point in $E$ 
near which $f^*$ happens to be continuous. 
Hence, by Lemma\,\ref{lem21}, $f^*$ is not continuous at $0$,
because $f\circ\alpha$ is unbounded on a sequence 
in $\cW$ which converges to $0$. 
This completes the proof of 
Theorem \ref{main}.
\end{proof}
\section{The holomorphic extension property}\label{sec4}
Let $M$ be a $CR$ submanifold, of $CR$ dimension $m$, and 
$CR$ codimension $d$, of a $\nuup$-dimensional 
complex manifold $X$. This means that $M$ is a smooth real submanifold of
$X$ and $T^{0,1}M=T^{0,1}X\cap\mathbb{C}TM$. \par
Let $p_0\in{M}$ and let
$(W;z_1,\hdots,z_{\nuup})$ be any coordinate neighborhood
in $X$ centered at $p_0$.
If $m+d=\nuup$, then the embedding 
$M\hookrightarrow{X}$ is generic and the coordinate neighborhood
$(W;z_1,\hdots,z_{\nuup})$  provides a generic 
$CR$-embedding of a neighborhood $U$ of $p_0$ in $M\cap{W}$ into an
open neighborhood of $0$ in $\mathbb{C}^{\nuup}$. If
$m+d=n<\nuup$, we can reorder the coordinates 
$z_1,\hdots,z_{\nuup}$ in such a way that $dz_1(p_0),\hdots, dz_n(p_0)$ are
linearly independent. Then the map 
$\phiup:p\mapsto\phiup(p)=(z_1(p),\hdots, z_n(p))$ yields again a generic 
$CR$-embedding of a neighborhood $U$ of $p_0$ in $M\cap{W}$ into an
open neighborhood of $0$ in $\mathbb{C}^{\nuup}$. 
We get
\begin{thm}\label{lem31} For each 
$a\in\{-\infty,0,\infty\}$ the following are equivalent:
\begin{enumerate}
\item\label{31a} 
the restriction map $\cO_{X,(p_0)}\to\cO^{a}_{M,(p_0)}$ is onto;
\item\label{31b}  the map $\phiup^*:\cO_{\mathbb{C}^n,(p_0)}\to\cO^{a}_{M,(p_0)}$ 
is an isomorphism.
\end{enumerate}
\end{thm}
\begin{proof}
The equivalence is a consequence of Theorem\,\ref{main}.  
In fact \eqref{31a} implies $CR$-hypoellipticity at $p_0$ and this,
by Theorem\,\ref{main}, implies \eqref{31b}. The inference
\eqref{31b}$\Rightarrow$\eqref{31a} is obvious.
\end{proof}
Moreover, we obtain 
\begin{thm}
  Assume that $M$ is a $CR$ submanifold of a complex manifold $X$
and $p_0\in{M}$.
Then the following are equivalent:
\begin{enumerate}
\item\label{11} 
the restriction map $\cO_{X,(p_0)}\to\cO_{M,(p_0)}^{\infty}$ is onto;
\item\label{12} 
the restriction map $\cO_{X,(p_0)}\to\cO_{M,(p_0)}^{0}$ is onto;
\item\label{13} 
the restriction map $\cO_{X,(p_0)}\to\cO_{M,(p_0)}^{-\infty}$ is onto.
\end{enumerate}
\end{thm}
\begin{proof} Since the statement is local, using 
Theorem\,\ref{lem31},
we can as well assume that $M$ is a 
generic $CR$ submanifold of an open ball
in $\mathbb{C}^n$, centered  at $p_0\! =\! 0$.  \par
To show that \eqref{11}$\Rightarrow$\eqref{12} it suffices to prove
that, for every compact $K\subset{M}$ containing a neighborhood 
of $0$ in $M$, the polynomial hull 
\begin{equation*}
  \hat{K}=\{z\in\C^n\!\mid\; |f(z)|\leq{\sup}_K|f|,\;
\forall f\in\C[z_1,\hdots,z_n]\}
\end{equation*} of $K$ in $\mathbb{C}^n$
contains a neighborhood of $0$ in $\mathbb{C}^n$. The implication will
indeed follow then by the approximation theorem in \cite{ba-tr}. \par
Let $\ring{K}$ be the interior in $M$ of an arbitrarily 
fixed compact neighborhood
$K$ of $0$ in~$M$. 
\par
For $r>0$, set $B_r=\{z\in\C^n\!\mid\; |z|<r\}$.
Fix $r>0$ in such a way that 
\mbox{$B_r\cap{M}$} is contained in some coordinate neighborhood 
$(U;t_1,\hdots,t_{2m+d})$ in $M$, with
$U\subset \ring{K}$. Then for every
integer $k$, the set
\begin{equation*}
  \mathbb{F}_k=\{(u,v)\in\cO^{\infty}_M(\ring{K})\times\cO(B_{r/2^k})\mid
v|_{M\cap{B_{r/2^k}}}=u|_{M\cap{B_{r/2^k}}}\}
\end{equation*}
is a closed subspace of 
the product $\cO^{\infty}_M(\ring{K})\times\cO(B_{r/2^k})$, endowed with
its standard Fr\'echet topology, and hence a
Fr\'echet space. 
The projection into the
first coordinate defines continuous linear maps $\piup_k:\mathbb{F}_k
\to \cO^{\infty}_M(\ring{K})$. By the assumption, 
${\bigcup}_{k}\piup_k(\mathbb{F}_k)=\cO^{\infty}_M(\ring{K})$. Hence some
$\piup_{{\nu}}(\mathbb{F}_{{\nu}})$ is of the second Baire
category. Then $\piup_{{\nu}}:\mathbb{F}_{{\nu}}\to\cO^{\infty}_M(\ring{K})$
is surjective and open by the Banach-Schauder theorem and we get:
\begin{equation*}
\left\{\begin{gathered}
  \exists C>0,\;\ell\in\mathbb{Z}_+,\; K'\Subset{K}\quad\text{such that}\quad
\forall u\in\cO^{\infty}_M(\ring{K})\;\;\exists\, \tilde{u}\in\cO_{\mathbb{C}^n}(
B_{2^{-{\nu}}r})\\
\text{with}\;\; \tilde{u}|_{M\cap{B_{r/2^{{\nu}}}}}=u|_{M\cap{B_{r/2^{{\nu}}}}},
\;\text{and}\;\;
{\sup}_{B_{r/2^{{\nu}+1}}}|\tilde{u}|\leq C\,\|u\|_{\ell,K'}=
{\sup}_{K'}{\sup}_{|\alpha|\leq\ell}\left|\tfrac{\partial^{|\alpha|}u}{\partial{t}^{\alpha}}
\right|.
\end{gathered}\right.
\end{equation*}
For $\epsilon>0$ set $K_{\epsilon}=\{z\in\mathbb{C}^n\mid \sup_{z'\in{K}}|z-z'|\leq
\epsilon\}$. By Cauchy's inequalities,
there is a positive constant $C_{\epsilon}$ such that 
\begin{equation*}
  \|f\|_{\ell,K'}\leq C_{\epsilon} {\sup}_{K_{\epsilon}}|f|,\quad
\forall f\in\mathbb{C}[z_1,\hdots,z_n]. 
\end{equation*}
This implies that 
\begin{equation*}
{\sup}_{B_{r/2^{{\nu}+1}}}|f|\leq C\,C_{\epsilon}{\sup}_{K_{\epsilon}}
|f|,\quad \forall f\in\mathbb{C}[z_1,\hdots,z_n].
\end{equation*}
An application of 
this inequality to the powers $f^h$ of the holomorphic polynomials
shows that in fact 
\begin{equation*}
{\sup}_{B_{r/2^{{\nu}+1}}}|f|\leq {\sup}_{K_{\epsilon}}
|f|,\quad \forall f\in\mathbb{C}[z_1,\hdots,z_n],
\end{equation*}
i.e. that $B_{r/2^{{\nu}+1}}$ is contained in the polynomial hull 
$\hat{K}_{\epsilon}$ of $K_{\epsilon}$. 
Since $\hat{K}={\bigcap}_{\epsilon>0}\hat{K}_{\epsilon}$,
the polynomial hull 
$\hat{K}$ contains
$B_{r/2^{{\nu}+1}}$. This completes the proof of 
\eqref{11}$\Rightarrow$\eqref{12}.
\par
To prove the implication \eqref{12}$\Rightarrow$\eqref{13} we use 
the elliptic partial
differential operator introduced in \cite{ba-tr} 
(see also \cite[Ch.II]{t92}). This is constructed in the following way.
We can assume that $dz_1,\hdots,dz_n,d\bar{z}_1,\hdots,d\bar{z}_m$ define
a maximal set of independent differentials on a neighborhood $U$ of
$0$ in $M$. Then we uniquely define commuting smooth complex vector fields
$L_1,\hdots,L_n,Z_1,\hdots,Z_m$ on $U$ by requiring that 
\begin{equation*}
  L_iz_j=\delta_{i,j},\; L_i\bar{z}_k=0,\; Z_hz_j=0,\; Z_h\bar{z}_k=\delta_{h,k},
\quad\text{for $1\leq{i,j}\leq{n}$, $1\leq{h,k}\leq{m}$}.
\end{equation*}
Then, for a large $c\in\R$, 
\begin{equation}\label{eq:26}
  \Delta_{L,cZ}={\sum}_{i=1}^nL_i^2+c^2{\sum}_{h=1}^mZ_h^2
\end{equation}
is elliptic on a neighborhood of $0$ in $M$, that, after shrinking, we can take
equal to~$U$. If $f\in\cO_{\mathbb{C}^n}(W)$ for an open neighborhood
$W$ of $0$ in $\mathbb{C}^n$ and $\nu$ is a non negative integer, then
\begin{equation*}
  \Delta_{L,cZ}^{k}f|_{U\cap{W}}=\left.\left(
\left({\sum}_{i=1}^n\tfrac{\partial^2}{\partial{z}_i^2}\right)^k
f\right)\right|_{U\cap{W}}.
\end{equation*}\par
In \cite{ba-tr} the following is proved
\begin{lem}\label{lem22}
  There is an open neighborhood $U'$ of $0$ in $U$ such that for every
$u\in\cO^{-\infty}_M(U)$ there is $w\in\cO^0_M(U')$ and an integer
$k\geq{0}$ such that
\begin{equation}\label{eq:27}
\qquad\qquad \qquad\qquad \qquad\qquad\;\;\,
 u|_{U'}=\Delta_{L,cZ}^{k}w.\qquad\qquad \qquad\qquad \qquad\qquad\;\;\,\qed
\end{equation}
\end{lem}
Let $u\in\cO^{-\infty}_M(U)$. By Lemma\,\ref{lem22} there is $w\in\cO_M^0(U')$
satisfying \eqref{eq:27}. If \eqref{12} is valid, there is an open neighborhood
$W$ of $0$ in $\C^n$ and a holomorphic function
$\tilde{w}\in\cO_{\mathbb{C}^n}(W)$ such that $\tilde{w}|_{U'\cap{W}}=
w|_{U'\cap{W}}$. In view of \eqref{eq:26}, 
$\tilde{u}=\big({\sum}_{i=1}^n\tfrac{\partial^2}{\partial{z}_i^2}\big)^{k}
\tilde{w}$
is a holomorphic function in $W$ such that $\tilde{u}|_{U'\cap{W}}=
u|_{U'\cap{W}}$. This shows that \eqref{12}$\Rightarrow$\eqref{13}.
Since the implication \eqref{13}$\Rightarrow$\eqref{11} is trivial,
the proof is complete.
\end{proof}
As a corollary of Lemma\,\ref{lem22}, we also state the following
regularity result, which will be useful to apply \cite{AHNP1}
to obtain holomorphic extension.
\begin{cor}\label{lem32a} Let $M$ be a $CR$ submanifold of a complex
manifold $X$, $p_0\in{M}$ and assume that all germs 
$\alphaup\in\cO_{M,(p_0)}^{-\infty}$ which are in $L^2_{\text{loc}}$ at $p_0$
are in $\cO_{M,(p_0)}^{\infty}$. Then $\cO_{M,(p_0)}^{-\infty}=\cO_{M,(p_0)}^{\infty}$.
\qed
\end{cor}
\section{Wedge extension and the wave front set}\label{sec5}
Theorem\,\ref{main} relates holomorphic extension to $\cC^\infty$-regularity.
Here we make a few remarks relating holomorphic 
wedge extension to the $\cC^\infty$
wave front set. 
For extension to open wedges attached to $M$, 
it is known that the directions of extension are nicely reflected
by the \textit{analytic} wave front set, which 
provides information on the extension of any {\it individual}
$CR$ distribution. Below we will see that local properties for {\it simultaneous} 
extension are related to the $\cC^\infty$-wave front sets of \textit{all} the
elements in $\cO^{-\infty}_{M,(p_0)}$. \par
Let $HM$ be the subbundle of the tangent bundle $TM$ consisting of the
real parts of vectors in $T^{0,1}M$. \par
For a point $p$ of a smooth $CR$ manifold $M$, 
we denote by $\tO_M(p)$ the $CR$ orbit of $p$ in $M$, i.e. the set of
all points of $M$ that can be linked with $p$ by a 
piecewise smooth curve with velocity vectors in $HM$.
A fundamental
result of Sussmann (\cite{Su73}) tells that each $CR$ orbit $\tO_M(p)$
is a smooth $CR$ submanifold, which turns out to have 
the same $CR$ dimension of $M$. 
If $U$ is an open neighborhood of $p$ in $M$ we can consider the
orbit $\tO_U(p)$. Clearly, if $p\in{V}^{\text{open}}\subset{U}^{\text{open}}
\subset{M}$, then $\tO_V(p)\subset\tO_U(p)$. The family of
$CR$ orbits $\tO_U(p)$, 
for $p\in{U}^{\text{open}}\subset{M}$, indexed by 
the filter
of open neighborhoods of $p$, 
uniquely defines a \textit{germ} of $CR$ manifold 
$\tO_{M,loc}(p)$, which is called \emph{the local $CR$ orbit} of $p$.
Tumanov's theorem in \cite{Tu1} 
yields local holomorphic extension to open wedges if $\tO_{U}(p_0)$ is open
(see also \cite{ba-ro, J96, Me94, me-po1, tr1}). 
More generally, the dimension of 
$\tO_{U}(p_0)$ can be related 
to the maximal number of
independent directions of $CR$ extension \cite{Tu2}.
\par
\smallskip
Denote by $\cD'(U)$, for $U^{\text{open}}\subset{M}$, the space of 
complex valued distributions
in $U$, and by  $\wf(u)\subset{T}^*U$ the wave front set of 
$u\in{\cD}'(U)$. 
For basic definitions and a thorough 
introduction to this topic we refer to \cite{H}. 
Recall that $\wf(u)$ is a closed conical
subset of $\dot{T}^* U$, which is the cotangent bundle
deprived of its zero section. It will be convenient to us to consider
also $\awf(u)=\wf(u)\cup\underline{0}$, where $\underline{0}$ is the
zero section of $T^*M$.  \par
If
$U^{\text{open}}\subset{M}$, 
and $u\in\cO^{-\infty}(U)$, then $\wf(u)\subset{H}^0M$,
where
$H^0M=\{\xiup\in{T}^*M\mid \xiup(v)=0,\;\forall v\in{H}_{\piup(\xi)}M\}$
is the \emph{characteristic bundle}
of the tangential $CR$ system.\par 
We prove the following
\begin{thm}\label{wf}
Let $M$ be a $CR$ submanifold, of $CR$ dimension $m$ and $CR$ codimension $d$,
of a complex manifold $X$, and $p_0\in{M}$. Then the following are equivalent
\begin{enumerate}
\item\label{wf1} $\dim_{\mathbb{R}}\tO_{M,\mathrm{loc}}(p_0)=2m+k$ ($0\leq{k}\leq{d}$);
\item\label{wf2} 
there is a $CR$ distribution $u$, defined on an open neighborhood 
$U$ of $p_0$, such that $\awf(u)\cap{T}^*_{p_0}M$ contains a 
$(d\! -\! k)$-dimensional $\R$-linear subspace, and $k$ is the smallest
integer with this property.
\end{enumerate}
\par
Assume that \eqref{wf1} holds true and that $\tO_{M,\mathrm{loc}}(p_0)$ does not
have the holomorphic extension property at $p_0$. Then there exists 
a $CR$ distribution $u$, defined on an open neighborhood 
$U$ of $p_0$, such that $\awf(u)\cap{T}^*_{p_0}M$ properly contains a 
$(d\! -\! k)$-dimensional $\R$-linear subspace.
\end{thm}
\begin{rmk}
Tumanov's theorem (see \cite{Tu1}) can be restated by saying
that all $CR$ functions defined on any fixed neighborhood of 
$p_0$ admit a
holomorphic extension to an open wedge with edge containing $p_0$ 
if and only 
if no $CR$ distribution $u$ has a $\awf(u)$ which contains
a real line of $T_{p_0}^* M$. 
Theorem\,\ref{wf} can be considered a generalization of that result to
the non minimal case.
\end{rmk}\begin{proof} 
We can assume that $M$ is a generic $CR$ submanifold
of $\mathbb{C}^n$.\par
Let
$\dim_{\mathbb{R}}\tO_{M,loc}(p_0)=2m+k$. Fix an open neighborhood $U$ of
$p_0$ in $M$.  Then there are generic $CR$ manifolds with boundary
$M_1,\ldots,M_k$ in $\mathbb{C}^n$, of dimension $2m\! +\! k\! +\! 1$, 
attached to $M$ along their boundaries 
near $p_0$, and such that every continuous $CR$ function $u$ on $U$
uniquely 
extends to each $M_j$ as a $CR$ function, continuous up to the boundary. 
Moreover the $M_j$ can be chosen so
that there are linearly independent vectors 
$X_1,\ldots,X_k\in T_{p_0}M\backslash T_{p_0}^c M$ such that $JX_j$
points into $M_j$. Then a 
standard deformation argument shows that 
for any continuous $CR$ function $u$ on $U$, 
$\wf(u)$ 
is contained in
$\{\xiup\in{H}^0M\mid \xiup(X_j)\geq{0}\}$ 
(this was observed in \cite{Tr2} 
for the larger analytic wave front set), 
so that
$\awf(u)$
cannot contain any 
$\R$-subspace of dimension larger than $d\! -\! k$. \par
To treat the case of a general
$CR$ distribution $u$, we utilize \cite{ba-tr}. There it is shown that
$u=(\Delta_{L+cZ})^q g$ on an open neighborhood $U'$ of $p_0$ in $M$,
where $\Delta_{L+cZ}$ 
is an appropriate second-order differential operator with smooth coefficients,
$q$ a sufficiently large positive integer and $g$ a continuous $CR$ function.
Since $\wf(u)\subset\wf(g)$, the fact that 
$\awf(u)\cap{T}^*_{p_0}M$ does not contain any
$\R$-subspace of dimension larger than $d\! -\!{k}$ follows from 
the case of continuous $CR$ functions.
\par
On the other hand, assume that there is a $CR$ submanifold $N$ of an open
neighborhood $U$ of $p_0$ in $M$, with the same $CR$ dimension $m$ and
$p_0\in{N}$. By taking $U$ small, we can find a $CR$ distribution 
on $U$
carried by $N$. \par
Indeed: When $N$ is open, there is nothing to prove.
If $N$ has smaller dimension, we fix
a positive measure $\mu$ with smooth density on $N$.  
A construction in \cite{ba-ro} yields a function ${v}$ 
which is $\cC^\infty$-smooth in a neighborhood of
$p_0$ in $N$, with ${v}(p_0)=1$, and such that
\begin{equation}\label{V}
T_N[\phi]=\int_N {v}\phi\,d\mu,\,\quad \phi\in\cD(U),
\end{equation}
is a $CR$ distribution on a possibly smaller 
neighborhood $U$ of $p_0$ in $M$.
In this case we obtain 
$\wf(u)\cap T_{p_0}^* M=(T_{p_0} N)^\perp$.
This completes the proof of the implication
\eqref{wf1}$\Rightarrow$\eqref{wf2}. The argument also shows that,
if there is a $CR$ distribution $u$, defined on a neighborhood 
$U$ of $p_0$, such that $\awf(u)\cap{T}^*_{p_0}M$ contains
an $\ell$-dimensional $\R$-subspace, then
$\dim_{\mathbb{R}}\tO_{M,\text{loc}}(p_0)\leq{2m\! +\! d\! -\!\ell}$. Thus we obtain
also the opposite implication \eqref{wf2}$\Rightarrow$\eqref{wf1}.
\par
Let us turn to the proof of the last statement. 
If $\tO_{M,loc}(p_0)$ is open, it is a consequence of Theorem \ref{main},
because a distribution $u$ with
$\wf(u)\cap T_{p_0}^* M=\emptyset$ is smooth near $p_0$. If $\tO_{M,loc}(p_0)$ 
is lower-dimensional,
we fix a $CR$ isomorphism $\pi:N\rightarrow N'\subset\C^{n'}$ from $N$ 
to a generic $CR$ manifold in
some lower-dimensional space. As explained 
before Lemma\,\ref{lem31}, we may assume that 
$\pi$ is induced by the
projection of $\C^n$ onto the complex subspace 
$\C^{n'}$ 
of the first $n'$ coordinates $z_1,\ldots,z_{n'}$. 
The Baouendi-Treves approximation
theorem says that there is a measure $\mu'$ on $N'$, 
with a smooth density on $N'$, 
such that any $CR$ distribution $S$ 
on $N'$ can be approximated
by polynomials $Q(z_1,\ldots,z_{n'})$, in the sense that
\begin{equation}\label{Q}
\int_{N'} Q_j\phi\,d\mu'\rightarrow S[\phi],\quad\forall\phi\in\cD(U'),
\end{equation}
holds on an appropriate neighborhood $U'\subset N'$ of $0=\pi(p_0)$. 
We can 
choose $\mu=\pi^*\mu'$ in (\ref{V}). We have the following Lemma.
\begin{lem}\label{pullback}
There is a neighborhood $U\subset M$ of $p_0$ such that for is any $CR$ 
distribution $u$ on $N'$
the formula
\begin{equation}\label{T_u}
T_u[\phi]=u[({v}\phi)\circ\pi^{-1}],\quad\forall\phi\in\cD(U),
\end{equation}
defines a $CR$ distribution $T_u$ on $U$ with support contained in $N\cap U$.
\end{lem}
\begin{proof}[Proof of Lemma \ref{pullback}]
Let $\{Q_j=Q_j(z_1,\ldots,z_{n'})\}$ be
a sequence of polynomials
approximating $u$ on some 
neighborhood $U'$ of $0$ in $N'$, as in (\ref{Q}). Since $\mu=\pi^*\mu'$, 
the distributions
$Q_j T_N:\phi\mapsto\int_N Q_j {v} \phi\,d\mu$ approximate the distribution 
in (\ref{T_u}), provided we
take $\phi$ with support in an open $U\subset M$ with 
$p_0\in{U}\cap N\subset\pi^{-1}(U')$. 
Being the products of
a $CR$ distribution by the restriction to $U$ 
of holomorphic functions, the 
$Q_j T_N$ are $CR$ distributions on $U$, and therefore also their
limit in the sense of distributions is a $CR$ distribution on $U$.
This completes the proof of the lemma.
\end{proof}
Since $N'$ does not have the extension property, by Theorem \ref{main} 
there is a
$CR$ distribution $\tilde{u}$
with $\wf_{N'}({\tilde{u}})\cap T_{0}^* N'\not=\emptyset$. 
It remains to check that 
$\wf({T_{\tilde{u}}})$ has the
desired properties. \par
To this purpose, we introduce smooth coordinates 
$(s_1,\ldots,s_{2m+k},t_1,\ldots,t_{\ell})$,
$\ell=d-k$, centered at $p_0$, such that $N=\{t_1=0,\ldots,t_{\ell}=0\}$. 
The distribution $T_{\tilde{u}}$ is a tensor
product
\begin{equation*}
T_{\tilde{u}}=({v}g\tilde{u})\otimes\delta_t,
\end{equation*}
where $\delta_t$ is the Dirac delta in the $t$-variables and $g$ 
is a smooth nonvanishing function such that
$d\mu'=g\,ds_1\ldots ds_{2m+k}$. Since ${v}(p_0)=1$, we can assume after
shrinking that ${v}\neq{0}$ on $U$. Then 
$\wf({u^*{v}g})=\wf({u^*})$ and the general rule to compute 
the wave front set of a tensor product \cite[Theorem 8.2.9]{H} yields
\begin{equation}
\wf({T_{\tilde{u}}})\cap T^*_{p_0}M=
\big(\awf_N({u^*})\times\langle dt_1,\ldots,dt_\ell\rangle\big)
\setminus\{(0,0)\},
\end{equation}
The proof is complete.
\end{proof}
\section{Some subellipticity conditions}\label{sec6}
In this section we recall some results of \cite{AHNP1} that are
relevant for our applications. In the following, $M$ is an abstract
$CR$ manifold, $\Zt(M)=\Gamma(M,T^{0,1}M)$ 
is the distribution of complex vector fields
of type $(0,1)$ on $M$, and ${\Ht}_{{}}(M)=\Gamma(M,HM)$ 
the distribution of the \textit{real} vector fields which are 
real parts elements of  $\Zt(M)$.
\subsection{The system
$\Theta_{{}}(M)$}
\begin{dfn} \label{def:a1} 
Set 
\begin{equation}\label{eq:a1}
 \Theta_{{}}(M)=\left\{Z\in\Zt(M)\left|{\begin{matrix}
\exists r\geq0,\; \exists Z_1,\hdots,Z_r\in\Zt(M),\;
\text{s.t.}\\
i[Z,\bar{Z}]+i\sum_{j=1}^r[Z_j,\bar{Z}_j]
\in{\Ht}_{{}}(M)
\end{matrix}}\right\}\right. .
\end{equation}
We denote by ${\At}_{{}}(M)$ the Lie subalgebra
of $\mathfrak{X}(M)$ generated by the real parts of
vectors in $\Theta_{{}}(M)$. If 
${\Ht}'_{{}}(M)=
\{\mathrm{Re}\,Z\mid Z\in\Theta_{{}}(M)\}$,
\begin{equation*}
  {\At}_{{}}(M)={\Ht}'_{{}}(M)+
[{\Ht}'_{{}}(M),
{\Ht}'_{{}}(M)]+[{\Ht}'_{{}}(M),
[{\Ht}'_{{}}(M),{\Ht}'_{{}}(M)]]+\cdots
\end{equation*}
\end{dfn}
We showed in \cite[Lemma 2.5]{AHNP1} that:
\begin{prop} With the notation introduced above,
$\Theta(M)$ is a left $\mathcal{C}^{\infty}(M)$-submo\-dule of
$\mathfrak{X}^{\mathbb{C}}(M)$. For every $Z\in\Theta_{{}}(M)$ and
every relatively compact open subset $U$ of $M$ there are a finite set
$Z_1,\hdots,Z_r$ of vector fields in $\Zt(M)$ and a constant
$C>0$ such that
\begin{equation*}
  \|\bar{Z}u\|_0^2\leq C\big( \|u\|_0^2+{\sum}_{i=1}^r\|Z_iu\|_0^2\big),\quad
\forall u\in\mathcal{C}^{\infty}_0(U).
\end{equation*}
\end{prop}
Hence, by \cite[Corollary 1.15]{AHNP1}, we obtain
\begin{thm}\label{thm:a3}
Let ${\Mt}_{{}}(M)$ be the 
${\At}_{{}}(M)$-Lie 
submodule of $\mathfrak{X}(M)$
generated by ${\Ht}_{{}}(M)$:
\begin{equation}
  \label{eq:a2}\begin{aligned}
 {\Mt}_{{}}(M)={\Ht}_{{}}(M)+
[{\At}_{{\Zt}}(M),{\Ht}_{{}}(M)]
\qquad\qquad\qquad\\
+
[{\At}_{{\Zt}}(M),
[{\At}_{{\Zt}}(M),{\Ht}_{{}}(M)]]+\cdots
\end{aligned}
\end{equation}
If
\begin{equation}\label{eq:a3}
  \{X_{p_0}\mid X\in{\Mt}_{{}}(M)\}=T_{p_0}M,
\end{equation}
then 
the system $\Zt(M)$ is subelliptic at $p_0$. This means that there
exists an open neighborhood $U$ of $p_0$ in $M$, vector fields
$Z_1,\hdots,Z_n\in\Zt(M)$, and constants $C,\varepsilon>0$ such that
\begin{equation}
  \label{eq:a4}
  \|u\|_{\varepsilon}^2\leq C\big(\|u\|_0^2+{\sum}_{i=1}^n\|Z_iu\|_0^2\big),\quad
\forall u\in\mathcal{C}^{\infty}_0(U).
\end{equation}
\end{thm}
\subsection{The system ${\Kt}_{{}}(M)$}
Under a certain constant rank assumption on $\Zt(M)$, we can give a more
explicit description of $\Theta_{{}}(M)$.
\begin{dfn}
The \emph{characteristic bundle} $H_{{}}^0M$
of $\Zt(M)$ is the
set 
of
\textit{real} covectors $\xi$ with 
$\langle{Z},{\xi}\rangle=0$ for all $Z\in
{\Zt}(M)$.\par
The  \emph{scalar Levi form} at $\xi\in{H}^0_{p}M$ is 
the Hermitian symmetric form
\begin{equation}
  \label{eq:a5}
  \mathfrak{L}_{\xi}(Z_1,\bar{Z}_2)=i\xi([Z_1,\bar{Z}_2]) \quad
\text{for}\quad Z_1,Z_2\in\Zt(M).
\end{equation}
The value of the right hand side of \eqref{eq:a5}
only depends on the values
$Z_1(p)$, $Z_2(p)$ of  $Z_1,Z_2$
at the base point $p=\pi(\xi)$.
Thus \eqref{eq:a5} is a Hermitian symmetric form on $T^{0,1}_pM$. 
Set:
\begin{gather}
  \label{eq:a6}
  H^{\oplus}_{{}}M
=\left\{\xi\in{H}^0_{{}}M\mid\mathcal{L}_{\xi}\geq{0}\right\},
\\
  \label{eq:a7}
  {\Kt}_{{}}(M)=\{Z\in\Zt(M)\mid\mathcal{L}_{\xi}(Z,
\bar{Z})=0,\;\forall\xi\in{H}^{\oplus}_{{}}M\},\\
\label{eq:a8}
\mathrm{K}_{{}}M=\dot{\bigcup}_{p\in{M}}\mathrm{K}_{p}M
\;\text{with}\;\mathrm{K}_{p}M=\{Z_p\mid Z\in
\mathbb{Z}_{{}}(M)\}. 
\end{gather}
\end{dfn}
We have (see \cite[Proposition 2.13]{AHNP1})
\begin{prop}\label{prop:a5}
${\Kt}_{{}}(M)$ is a left $\mathcal{C}^{\infty}(M)$
submodule of $\Theta_{{}}(M)$.
Assume in addition that 
$H^{\oplus}_{{}}M$ and $\mathrm{K}_{{}}M$ 
are smooth vector bundles on $M$. Then 
\begin{equation}
  \label{eq:a9}
{\Kt}_{{}}(M)=\Theta_{{}}(M).
\end{equation}
\end{prop}
\subsection{Hypoellipticity} \cite{al07}
Subelliptic estimates imply
regularity. We have indeed (see \cite[Theorem 4.1]{AHNP1},
\cite[Theorem 4.3]{HN00}):
\begin{thm} Let $M$ be an $m$-dimensional smooth manifold.
Let $U$ be an open subset of $M$, and
$Z_1,\hdots,Z_n$ complex vector fields on $U$ such that,
for some positive constants $C,\epsilon>0$ 
\eqref{eq:a4} is valid.
If $u\in{L}^2_{loc}$, 
$a_i\in{L}^{\infty}_{loc}(U)$, $f_i\in{L}^2_{loc}(U)$ for $i=1,\hdots,n$
satisfy
\begin{equation}
  \label{eq:a10}
  Z_iu+a_iu=f_i,\quad\text{for}\; i=1,\hdots,n\quad\text{on}\;{U},
\end{equation}
then :
\begin{enumerate}
\item $u\in{W}^{\epsilon}_{loc}(U)$;
\item if $0<s\leq \tfrac{m}{2}$, $a_i\in\mathcal{C}^s(U)$ and
$f_i\in{W}^{s}_{loc}(U)$, then $u\in{W}^{s+\epsilon}_{loc}(U)$;
\item if $s>\tfrac{m}{2}$, $a_i\in{W}^s_{loc}(U)$ and
$f_i\in{W}^{s}_{loc}(U)$, then $u\in{W}^{s+\epsilon}_{loc}(U)$;
\item in particular, if $a_i\in\mathcal{C}^{\infty}(U)$, 
$f_i\in{W}^{s}_{loc}(U)$, then $u\in{W}^{s+\epsilon}_{loc}(U)$.
\end{enumerate}
Here we indicate by $W^s_{loc}(U)$ the $L^2$-Sobolev space of 
order $s$.
\end{thm}
Then we obtain from Lemma\,\ref{lem32a}:
\begin{cor}
  If \eqref{eq:a3} holds true, then $\tO_{M,(p_0)}^{-\infty}=\tO_{M,(p_0)}^{\infty}$.
\end{cor}
\section{Examples}\label{sec7}
A large class of examples of $CR$ submanifolds of complex manifolds is
provided by the orbits of the real forms in complex flag manifolds.
We recall that a complex flag manifold
is a compact homogeneous space $X$ of a semisimple complex Lie group
$\mathbf{G}$. The isotropy of a point of $X$ is a \textit{parabolic}
subgroup $\mathbf{Q}$ of $\mathbf{G}$, i.e. a closed
connected subgroup whose Lie algebra $\mathfrak{q}$ contains a maximal
solvable Lie subalgebra $\mathfrak{b}$ of the Lie algebra $\mathfrak{g}$
of $\mathbf{G}$. If $\mathbf{G}_0$ is a \textit{real form} of
$\mathbf{G}$, i.e. a connected real Lie subgroup of $\mathbf{G}_0$
with Lie algebra $\mathfrak{g}_0$ such that 
\mbox{$\mathfrak{g}\!
=\!\mathfrak{g}_0\!\oplus\!\!{i}\mathfrak{g}_0$}, then $\mathbf{G}_0$
has finitely many orbits in $X$. In particular, there are open orbits and
a minimal orbit $M$ which is compact (see \cite{Wolf69}). 
The structure of the orbits only depend on the Lie algebras involved,
and are therefore completely determined by the pair 
$(\mathfrak{g}_0,\mathfrak{q})$, which is called a
\emph{$CR$ algebra}, consisting
of the Lie algebra of the real form $\mathbf{G}_0$
and of the Lie algebra of the parabolic
subgroup $\mathbf{Q}$.  \par
The embedding of $M$ in $X$ defines
a $CR$ structure on $M$. 
The minimal orbits are classified by their \emph{cross-marked Satake diagrams}.
A complete list of these diagrams is given e.g. in the appendix to \cite{AMN06}. 
Many properties of the minimal orbits are read off these diagrams:
minimality is equivalent to the fact that the corresponding
\textit{$CR$ algebra} $(\mathfrak{g}_0,\mathfrak{q})$ is fundamental and is 
described by \cite[Theorem\,9.3]{AMN06}. In \cite[\S{13}]{AMN06} 
all \textit{essentially pseudoconcave} minimal orbits
are classified in terms of their associated diagrams.
Since essential pseudoconcavity (see \cite{HN00}) implies
\eqref{eq:a3}, all these orbits are at every point $CR$-hypoelliptic and
therefore have the holomorphic extension property by Theorem\ref{main}.
Globally defined $CR$ functions on this class of $CR$ manifolds 
and their properties were considered
in \cite{al07}.\par
We give below some more explicit examples to illustrate this application.\par
Let $X$ be the complex flag manifold consisting of the 
flags
\begin{equation*}
 \ell_1\subset\ell_3\subset\cdots \subset\ell_{2k-1}\subset
\ell_{2k+2}\subset\cdots\subset\ell_{4k-2} \subset\C^{4k},
\end{equation*}
where $k$ is a positive integer and 
$\ell_i$ is a $\C$-linear subspace of dimension $i$ of $\C^{4k}$.
Let $M$ be the minimal orbit for the action of the group
$\mathbf{SU}(2k,2k)$ of complex 
$4\! k\!\times \! 4\! k$ matrices that leave invariant
a Hermitian symmetric form of signature $(2k,2k)$. Then $M$ has
$CR$ dimension $2k$ and $CR$ codimension $8k^2-6k-1$ and we need 
$2k$ commutators of $\Ht(M)$ to span $TM$ (these
numbers were computed in \cite{MN98}). However, $M$ is minimal
and essentially pseudoconcave and therefore is $CR$-hypoelliptic and
has the holomorphic extension property at all points.
\par
Another example is the minimal orbit of the special group $\mathbf{G}_0$
of type $E_6III$ corresponding to the cross-marked Satake diagram
\par\bigskip
{}
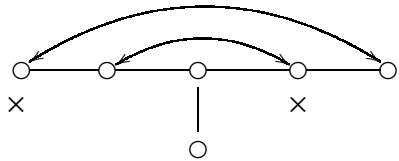
\par\medskip
\begin{equation*}
  \xymatrix@R=-.25pc{\medcirc\!\!\ar@{-}[r]\ar@{<->}@/^2pc/[rrrr]
&\!\!\medcirc\!\!\ar@{-}[r]\ar@{<->}@/^1pc/[rr]
&\!\!\medcirc\!\! \ar@{-}[ddddd]
\ar@{-}[r]
&\!\!\medcirc\!\!\ar@{-}[r]
&\!\!\medcirc\\
\\ \times &&&\times
\\
\\
\\
&& \medcirc
} 
\end{equation*}
It corresponds to a $CR$ manifold of $CR$ dimension $4$ and
$CR$ codimension $25$, with $6$ commutations needed to span
$TM$ from $\Ht(M)$. This is also essentially pseudoconcave and
therefore is $CR$-hypoelliptic and has the holomorphic extension property
at each point.
\bibliographystyle{amsplain}

\renewcommand{\MR}[1]{}
\bibliography{nwext}
\providecommand{\bysame}{\leavevmode\hbox to3em{\hrulefill}\thinspace}
\providecommand{\MR}{\relax\ifhmode\unskip\space\fi MR }
\providecommand{\MRhref}[2]{%
  \href{http://www.ams.org/mathscinet-getitem?mr=#1}{#2}
}
\providecommand{\href}[2]{#2}

\end{document}